\documentclass[12pt]{article}
\usepackage{amssymb}
\usepackage{amsmath}
\usepackage{amsfonts}
\usepackage{tikz}
\usepackage{bookmark}

\newtheorem{theorem}{Theorem}

\newtheorem{lemma}[theorem]{Lemma}

\newtheorem{proposition}[theorem]{Proposition}

\newenvironment{proof}[1][Proof]{\noindent\textbf{#1.} }{\ \rule{0.5em}{0.5em}}
\begin{document}

\title{Variations and extensions of the Gaussian concentration inequality, Part II}
\author{Daniel J. Fresen\thanks{%
University of Pretoria, Department of Mathematics and Applied Mathematics,
daniel.fresen@up.ac.za}}
\date{}
\maketitle

\begin{abstract}
We prove concentration inequalities for $f\left( X\right) $ about its median, where $X$ is a random vector in $\mathbb{R}^n$ with independent heavy tailed coordinates of Weibull or power type, and $f:\mathbb{R}^n\rightarrow\mathbb{R}$ is a locally Lipschitz function. This paper is part of a series of four papers, Part I, Part II and two supporting papers. It can be read independently of Part I.
\end{abstract}

\tableofcontents

\section{Results}

Let $Lip\left( T,x\right) $ denote the local Lipschitz constant of
a function around a point $x$,%
\begin{equation}
Lip\left( T,x\right) =\lim_{\varepsilon \rightarrow 0^{+}}Lip\left(
T|_{B\left( x,\varepsilon \right) }\right)  \label{loc lip def}
\end{equation}%
Associated to a real valued random variable $X_1$ are its distribution, CDF and quantile function defined as
\begin{eqnarray}
\mu_1(E)&=&\mathbb{P}\{X_1\in E\}, E\in\mathcal{B}(\mathbb{R})\nonumber\\
F_1(t)&=&\mathbb{P}\{X_1\leq t\}, t\in\mathbb{R}\nonumber\\
F_1^{-1}(s)&=&\inf \left\{ t\in \mathbb{R}:F_1(t)\geq s\right\}, s\in(0,1)\label{gen inv def 2016}
\end{eqnarray}%
where $\mathcal{B}(\mathbb{R})$ is the Borel sigma algebra. If $X_1 $ has a continuous
density $h=d\mu_1 /dx$ then for all points $s\in (0,1)$ and $t\in \mathbb{R}$,%
\begin{equation*}
Lip\left( F_1^{-1},s\right) =\frac{1}{h\left( F_1^{-1}\left( s\right) \right) }%
\hspace{0.65in}Lip\left( F_1^{-1}\circ \Phi ,t\right) =\frac{\phi \left(
t\right) }{h\left( F_1^{-1}\left( \Phi \left( t\right) \right) \right) }
\end{equation*}%
where we interpret $a/0=\infty$ for $a>0$,
\[
\phi(t)=\frac{1}{\sqrt{2\pi}}e^{-t^2/2}\hspace{1cm} and \hspace{1cm} \Phi(t)=\int_{-\infty}^t\phi(u)du
\]
For a differentiable function $f:\mathbb{R}%
^{n}\rightarrow \mathbb{R}$ and $1\leq s\leq \infty $ set%
\begin{equation}
Lip_{s}\left( f\right) =\sup_{x\in \mathbb{R}^{n}}\left\vert \nabla
f(x)\right\vert _{s}\hspace{0.65in}Lip_{s}^{\sharp }\left( f\right) =\left(
\sum_{i=1}^{n}\sup_{x\in \mathbb{R}^{n}}\left\vert \frac{\partial f}{%
\partial x_{i}}(x)\right\vert ^{s}\right) ^{1/s}  \label{measuring Lip p}
\end{equation}%
with the usual interpretation (involving $\max $) when $s=\infty $. In
general, $Lip_{s}\left( f\right) \leq Lip_{s}^{\sharp }\left( f\right) $.
There are two special cases when $Lip_{s}\left( f\right) =Lip_{s}^{\sharp
}\left( f\right) $. The first is when $f$ is linear. The second is when $%
s=\infty $. By considering a path integral, and a local linear
approximation, it follows that $Lip_{s}\left( f\right) $ is the Lipschitz
constant of $f$ with respect to the $\ell _{s^{\ast }}^{n}$ norm on $\mathbb{%
R}^{n}$, where $s^{\ast }=s/\left( s-1\right) $ when $1<s\leq \infty $
(setting $\infty /\left( \infty -1\right) :=1$).

In both Theorems \ref{conc sum weibull type} and \ref{poly tails Lip} below, the assumption of differentiability is a convenience rather than a necessity.

\subsection{Weibull tails}

\begin{theorem}[Weibull type tails $\sim \exp \left( -\left\vert
t\right\vert ^{q}\right) $, $0<q<1$]
\label{conc sum weibull type}Let $n\in \mathbb{N}$, $0<q<1$, and let $\left(
\mu _{i}\right) _{1}^{n}$ be a sequence of probability measures on $\mathbb{R%
}$, each with corresponding cumulative distribution $F_{i}$ and quantile
function $F_{i}^{-1}$ as in (\ref{gen inv def 2016}) such that for all $s\in 
\mathbb{R}$,%
\begin{equation}
Lip\left( F_{i}^{-1}\circ \Phi ,s\right) \leq \left( 1+\left\vert
s\right\vert \right) ^{-1+2/q}  \label{condition A rip}
\end{equation}%
Let $f:\mathbb{R}^{n}\rightarrow \mathbb{R}$ be differentiable and let $%
\left( X_{i}\right) _{1}^{n}$ be a sequence of independent random variables,
each with corresponding distribution $\mu _{i}$. Then for all $t>0$,%
\begin{equation}
\mathbb{P}\left\{ \left\vert f\left( X\right) -\mathbb{M}f\left( X\right)
\right\vert >t\right\} \leq 2\exp \left( -c_{q}\min \left\{ \left( \frac{t}{%
Lip_{2}^{\sharp }\left( f\right) }\right) ^{2},\left( \frac{t}{Lip_{\infty
}^{\sharp }\left( f\right) }\right) ^{q}\right\} \right)
\label{weibull variation}
\end{equation}%
and%
\begin{eqnarray}
&&\mathbb{P}\left\{ \left\vert f\left( X\right) -\mathbb{M}f\left( X\right)
\right\vert >C_{q}\left(\log \left(e+ \frac{n}{t^{-2+4/q}}\right)
\right)^{1/q-1/2} \left( tLip_{2}\left( f\right) +t^{2/q}Lip_{\infty }\left(
f\right) \right) \right\}  \notag \\
&\leq &2\exp \left( -t^{2}/2\right)  \label{nonlinear Weibull 2}
\end{eqnarray}
\end{theorem}

\bigskip

\textbf{Comments for Theorem \ref{conc sum weibull type}}

\bigskip

\noindent $\bullet $ \textit{Comment 1:} When $f$ is linear, say $f(x)=\sum_1^na_ix_i$, for non-zero $a\in \mathbb{R}^{n}$, and we assume without loss of generality that $\mathbb{E}X_{i}=0$,  then (\ref{weibull variation}) can be written as
\begin{equation}
\mathbb{P}\left\{ \left\vert \sum_{i=1}^{n}aX_{i}\right\vert >t\right\} \leq
2\exp \left( -c_{q}\min \left\{ \left( \frac{t}{\left\vert a\right\vert }%
\right) ^{2},\left( \frac{t}{\left\vert a\right\vert _{\infty }}\right)
^{q}\right\} \right)  \label{Weibul dist for 0<q<=1}
\end{equation}%
This is already known and holds under the slightly weaker tail assumption
\[
\mathbb{P}\left\{ \left\vert X_{i}\right\vert \geq t\right\} \leq 2\exp \left( -t^{q}\right), t>0
\]
There are several ways to show this. For example, one could use moment estimates of Hitczenko,
Montgomery-Smith and Oleszkiewicz \cite[Theorem 1.1]{HitMont1997} as modified in \cite{Fr20} (see also Lata\l a's paper \cite{Lat97}), or one could use a modified version of the Gaussian concentration inequality (either in the style of this paper, or using a method of restricting and extending functions). We refer the reader to \cite{Fr20} for a more detailed discussion of this.

\bigskip

\noindent $\bullet $ \textit{Comment 2:} A related nonlinear estimate was proved by Barthe, Cattiaux, and Roberto \cite[Example 5.4]{BaCatRo}. If $X$ has i.i.d. coordinates, each with distribution $\mu_i$ and corresponding density $d\mu_i/dx=\left(2\Gamma\left(1+1/q\right)\right)^{-1}\exp\left(-\left\vert x \right\vert^q\right)$, then using a weak Poincar\'{e} inequality they proved that
\begin{equation*}
\mathbb{P}\left\{ \left\vert f\left( X\right) -\mathbb{M}f\left( X\right)
\right\vert >t\right\} \leq 2\exp \left( -c_{q}\min \left\{ \frac{t}{Lip_{2}\left(
f\right) }\left(\log n \right) ^{1-1/q},\left( \frac{t}{%
Lip_{2}\left( f\right) }\right) ^{q}\right\} \right)
\end{equation*}%
Or equivalently,
\begin{eqnarray}
\mathbb{P}\left\{ \left\vert f\left( X\right) -\mathbb{M}f\left( X\right)
\right\vert >C_{q}\left(t^2\left( \log n \right)^{-1+1/q}+t^{2/q}\right)
Lip_{2}\left( f\right)  \right\} \leq 2\exp \left( -t^{2}/2\right)  \label{BaCatRo Weibull}
\end{eqnarray}
For $t$ below a certain value (\ref{BaCatRo Weibull}) is superior to (\ref{nonlinear Weibull 2}), however for large values of $t$ (\ref{nonlinear Weibull 2}) improves on (\ref{BaCatRo Weibull}) by replacing $Lip_{2}\left( f\right)$ with the potentially smaller quantity $Lip_{\infty}\left( f\right)$, which is the correct behaviour in the linear case.

\bigskip

\noindent $\bullet $ \textit{Comment 3:} The quantity%
\begin{equation*}
\left(\log \left(e+ \frac{n}{t^{-2+4/q}}\right)
\right)^{1/q-1/2}
\end{equation*}%
cannot be completely erased from (\ref{nonlinear Weibull 2}), otherwise this
could be written as%
\begin{equation*}
\mathbb{P}\left\{ \left\vert f\left( X\right) -\mathbb{M}f\left( X\right)
\right\vert >t\right\} \leq 2\exp \left( -c_{q}\min \left\{ \left( \frac{t}{%
Lip_{2}\left( f\right) }\right) ^{2},\left( \frac{t}{Lip_{\infty }\left(
f\right) }\right) ^{q}\right\} \right)
\end{equation*}%
which would imply that $\mathrm{Var}\left( \left\vert X\right\vert _{\infty
}\right) <C_{q}$, when in fact $\mathrm{Var}\left( \left\vert X\right\vert
_{\infty }\right) \approx \left( \log n\right) ^{-2+2/q}$. The current bound
gives the estimate $\mathrm{Var}\left( \left\vert X\right\vert _{\infty
}\right) \leq C_{q}\left( \log n\right) ^{-1+2/q}$ which is off by a factor
of $\log n$, exactly the same factor by which the classical Gaussian
concentration inequality is off by in the case when $q=2$ (specifically for
the $\ell _{\infty }^{n}$ norm). For certain functions it may be possible to
decrease the exponent of the logarithmic term, say from $1/q-1/2$ to $1/q-1$%
, using Talagrand's $L_{1}$-$L_{2}$ inequality (the Gaussian version as in 
\cite{CorLed}), or by applying the methodology presented here to $\eta \circ
f$ for an appropriate choice of $\eta $ to achieve a superconcentrated
estimate, which is a trick we have exploited in connection with Gaussian
concentration of the $\ell _{p}^{n}$ norm, see \cite{Fr22}. Or one could simply combine (\ref{nonlinear Weibull 2}) with (\ref{BaCatRo Weibull}).

\bigskip

\noindent $\bullet $ \textit{Comment 4:} In case $\left\vert \nabla f\left( X\right) \right\vert
_{s}$ is with high probability much smaller than $Lip_{s}\left( f\right) $,
including the case when $Lip_{s}\left( f\right) =\infty $ and/or when one
has a high probability bound on $\partial _{i}f\left( X\right) $, \thinspace 
$1\leq i\leq n$, one can prove variations of Theorem \ref{conc sum weibull
type} (using a similar proof) that take the distribution of these into
account.

\bigskip

\noindent $\bullet $ \textit{Comment 5:} For $q\in \left[ 1,2\right] $ related estimates can be
proved using results of Talagrand \cite[Theorem 2.4]{Tal94sup} and Gozlan 
\cite[Proposition 1.2]{Gozl}, which give bounds of the form%
\begin{equation*}
\mu ^{n}\left( A+tB_{2}^{n}+t^{2/q}B_{q}^{n}\right) \geq 1-\exp \left(
-Dt^{2}\right) :t\geq 0
\end{equation*}%
assuming that $\mu $ satisfies a type of Poincar\'{e} inequality on $\mathbb{%
R}$ with constant $C$, $\mu ^{n}$ is the $n$-fold product of $\mu $, and $D$
is a constant depending only on $C$. This includes the case $d\mu /dx=\left(
2\Gamma \left( 1+1/q\right) \right) ^{-1}\exp \left( -\left\vert
x\right\vert ^{q}\right) $. Taking $A=\left\{ x:f\left( x\right) \leq 
\mathbb{M}f\left( x\right) \right\} $, if $x\in
A+tB_{2}^{n}+t^{2/q}B_{q}^{n} $ then there exists $a\in A$, $u\in B_{2}^{n}$%
, $v\in B_{q}^{n}$ such that $x=a+tu+t^{2/q}v$, so $\left\vert f\left(
x\right) -f\left( a+tu\right) \right\vert \leq t^{2/q}Lip_{q^{\ast }}\left(
f\right) $ and $\left\vert f\left( a+tu\right) -f\left( a\right) \right\vert
\leq tLip_{2}\left( f\right) $, which implies $f\left( x\right) \leq \mathbb{%
M}f\left( x\right) +tLip_{2}\left( f\right) +t^{2/q}Lip_{q^{\ast }}\left(
f\right) $. Together with a similar lower bound,%
\begin{equation*}
\mathbb{P}\left\{ \left\vert f\left( X\right) -\mathbb{M}f\left( X\right)
\right\vert >t\right\} \leq 2\exp \left( -c_{q}\min \left\{ \left( \frac{t}{%
Lip_{2}\left( f\right) }\right) ^{2},\left( \frac{t}{Lip_{q^{\ast }}\left(
f\right) }\right) ^{q}\right\} \right)
\end{equation*}%

\subsection{Power tails}

\begin{theorem}
\label{poly tails Lip}There exists a universal constant $C>0$ such that the
following is true. Let $n\in \mathbb{N}$, $2<q<\infty $, $2q\left(
q-2\right) ^{-1}<p<\infty $, and let $\left( \mu _{i}\right) _{1}^{n}$ be a
sequence of probability measures on $\mathbb{R}$, each with corresponding
cumulative distribution $F_{i}$ and quantile function $F_{i}^{-1}$ as in (%
\ref{gen inv def 2016}) such that for all $s\in \left( 0,1\right) $,%
\begin{equation}
Lip\left( F_{i}^{-1},s\right) \leq \min \left\{ s,1-s\right\} ^{-1-1/q}
\label{Lip poly cond}
\end{equation}%
Let $f:\mathbb{R}^{n}\rightarrow \mathbb{R}$ be differentiable and let $%
\left( X_{i}\right) _{1}^{n}$ be a sequence of independent random variables,
each with corresponding distribution $\mu _{i}$. Then for all $t>0$,%
\begin{equation}
\mathbb{P}\left\{ \left\vert f\left( X\right) -\mathbb{M}f\left( X\right)
\right\vert >C_{q}t\left( Lip_{2}^{\sharp }\left( f\right) +e^{t^2/(2q)} Lip_{q}^{\sharp }\left( f\right)\right)\right\} \leq Ce^{-t^{2}/2}  \label{poly conc one}
\end{equation}%
where $C_{q}>0$ is a function of $q$, and%
\begin{equation}
\mathbb{P}\left\{ \left\vert f\left( X\right) -\mathbb{M}f\left( X\right)
\right\vert >C_{p,q}tLip_{p}\left( f\right) \left( n^{1/2-1/p}+n^{1/q}e^{t^2/(2q)} \right) \right\} \leq Ce^{-t^{2}/2}  \label{poly conc two}
\end{equation}%
where $C_{p,q}>0$ is a function of $\left( p,q\right) $.
\end{theorem}

\bigskip

\textbf{Comments for Theorem \ref{poly tails Lip}}

\bigskip

\noindent $\bullet$ \textit{Comment 1:} The following result is due to Barthe, Cattiaux and Roberto \cite[Example 5.3%
]{BaCatRo} (see Theorem 5.1 in their paper for a more general result).

\begin{theorem}
\label{BarCatRob}Let $\alpha>0$, let $X$ be a random vector in $\mathbb{R}^{n}$ with
probability density (with respect to Lebesgue measure) $d\mu /dx=\alpha
^{n}2^{-n}\prod_{i=1}^{n}\left( 1+\left\vert x_{i}\right\vert \right)
^{-1-\alpha }$, and let $f:\mathbb{R}^{n}\rightarrow \mathbb{R}$ satisfy $%
\left\vert f\left( x\right) -f\left( y\right) \right\vert \leq \left\vert
x-y\right\vert $ for all $x,y\in \mathbb{R}^{n}$. Then there exists $%
t_{0}\left( \alpha \right) >e$ and $C\left( \alpha \right) >0$ such that for
all $t\geq t_{0}\left( \alpha \right) $,%
\begin{equation*}
\mathbb{P}\left\{ \left\vert f\left( X\right) -\mathbb{M}f\left( X\right)
\right\vert >tn^{1/\alpha }\right\} \leq C\left( \alpha \right) \left( \frac{%
\log t}{t}\right) ^{\alpha }
\end{equation*}
\end{theorem}
Their bound can be written as
\begin{equation}
\mathbb{P}\left\{ \left\vert f\left( X\right) -\mathbb{M}f\left( X\right)
\right\vert >C_qt^2e^{t^2/(2q)}n^{1/q}\right\} \leq C_qe^{-t^2/2}\label{rewritiones}
\end{equation}
We have taken the liberty of including the $C_q$ in the deviation. For $t$ large enough so that
\[
te^{t^2/(2q)}> C_q n^{\frac{1}{2}-\frac{1}{p}-\frac{1}{q}}Lip_{p}\left( f\right)
\]
(\ref{poly conc two}) improves on (\ref{rewritiones}). Note that by the assumptions on $p$ and $q$, $\frac{1}{2}-\frac{1}{p}-\frac{1}{q}>0$ and $Lip_{p}(f)\leq Lip_{2}(f)$, which in their result is at most $1$. It is not unusual for $Lip_{p}(f)$ to be much smaller than $1$.

\bigskip

\noindent $\bullet$ \textit{Comment 2:} Whenever $f$ is linear, say $f(x)=\sum_1^n a_iX_i$ for $a\neq 0$, (\ref{poly conc one}) improves on (\ref{rewritiones}) for all $t$. In this case we may replace the basic assumption (\ref{Lip poly
cond}) with $\mathbb{P}\left\{ \left\vert X_{i}\right\vert >t\right\} \leq
c_{q}\left( 1+t\right) ^{-q}$. To do this, take an i.i.d. sequence $(Y_i)_1^n$ independent of $(X_i)_1^n$ with $\mathbb{P}\left\{ \left\vert Y_{i}\right\vert >t\right\} =\left( 1+t\right) ^{-q}$. One can show, as we did in the $3^{rd}$ arXiv version of \cite{Fr20IIa}, that both $(Y_i)_1^n$ and $(X_i+Y_i)_1^n$ satisfy (\ref{Lip poly cond}). Applying the result to $\sum a_i(X_i+Y_i)$ and to $\sum a_iY_i$ gives the result for $\sum a_iX_i$. However in the linear case, using a different approach, this can be improved for certain values of $t$ (see again \cite{Fr20IIa} and references therein).

\bigskip

\noindent $\bullet$ \textit{Comment 3:} Consider the special case of Theorem \ref{poly tails Lip} where $%
f(x)=n^{-1/2}\sum_{i=1}^{n}x_{i}$, and (\ref{poly conc one}) implies%
\begin{equation*}
\mathbb{P}\left\{ \left\vert f\left( X\right) -\mathbb{M}f\left( X\right)
\right\vert >C_{q}t\left( 1+n^{1/q-1/2}e^{t^2/(2q)}
\right) \right\} \leq C e^{-t^{2}/2}
\end{equation*}%
which is sub-Gaussian up to probability $Cn^{1-q/2}$ and matches the variance up to the factor $C_{q}$. Compare this with the following non-uniform version of the Berry-Esseen bound (see \cite[%
Chapter V \S 2-4 Theorems 3 and 13]{Pet75}).

\begin{theorem}
\label{non unif Berry Esseen}Let $r\in \left[ 3,\infty \right) $, and let $%
\left( X_{i}\right) _{1}^{n}$ be an i.i.d. sequence with $\mathbb{E}X_{1}=0$%
, $\mathbb{E}X_{1}^{2}=1$, and $\mathbb{E}\left\vert X_{1}\right\vert
^{r}<\infty $. Then for all $x\in \mathbb{R}$, 
\begin{equation*}
\left\vert \Phi (x)-\mathbb{P}\left\{ \frac{1}{\sqrt{n}}\sum_{i=1}^{n}X_{i}%
\leq x\right\} \right\vert \leq C_{r}\left( 1+\left\vert x\right\vert
\right) ^{-r}\left( n^{-1/2}\mathbb{E}\left\vert X_{1}\right\vert
^{3}+n^{-(r-2)/2}\mathbb{E}\left\vert X_{1}\right\vert ^{r}\right)
\end{equation*}%
\end{theorem}
Assuming that $\mathbb{E}\left\vert X_{1}\right\vert ^{r}<C_{r}$ and $%
n>C_{r} $, this gives a sub-Gaussian bound on $n^{-1/2}\sum X_i$ up to
probability $C_{r}\left( \log n\right) ^{-r/2}n^{-1/2}$. There is more that can be said here; the point is simply that as a nonlinear result Theorem \ref{poly tails Lip} does reasonably well in the linear case.

\bigskip

\underline{Notation and conventions}

$\mathbb{M}$ denotes median, $C$, $c$, $%
C^{\prime }$ etc. denote universal constants that may represent different values at each appearance, dependence on variables will usually be indicated by subscripts, $C_{q}$, $%
c_{q}$ etc. $\mathbb{N}$ denotes $\{1,2,3,...\}$. The term `sub-Gaussian bound' is used loosely to mean a bound where the probability is sub-Gaussian in terms of the deviation, at least for some values of $t$ (usually in a central region of the distribution). This differs from a sub-Gaussian random variable that would satisfy sub-Gaussian bounds for all $t$. Upper bounds (for probabilities) of the form $C\exp\left(-c_1t^2\right)$ can be replaced with $2\exp\left(-c_2t^2\right)$ by taking $c_2$ sufficiently small, and we do this without further explanation.

\section{\label{section concen ineq}Proofs}

\subsection{\label{outi}Outline of the general method}

We write $X_i=F_i^{-1}(\Phi(Z_i))$ where $Z=(Z_i)_1^n$ has the standard normal distribution in $\mathbb{R}^n$. The underlying probability space is of no relevence, and we may assume without loss of generality that such a $Z$ exists. Then $f(X)=\psi(Z)$ where $\psi(x)=f(F_i^{-1}(\Phi(Z_i))_1^n)$. We then apply Pisier's version of the Gaussian concentration inequality, see \cite{Pis0} and Proposition \ref{PisGau} below, to conclude that for all convex $\varphi:\mathbb{R}\rightarrow\mathbb{R}$,
\[
\mathbb{E}\varphi\left(\psi(Z)-\psi(Y)\right)\leq\mathbb{E}\varphi\left(Z'\left\vert \nabla \psi(Z)\right\vert \right)
\]
where $Y$ is an independent copy of $Z$ and $Z'$ is a random variable in $\mathbb{R}$ with the standard normal distribution independent of $Y$ and $Z$. By the theory of tail comparison inequalities under convex majorization, in particular by a result of Meilijson and N\'{a}das \cite{MeiNad} that is modified in \cite{Fr conv maj} to suite our purposes, see Proposition \ref{Gaussian vvv} below, we conclude that the quantiles of $\psi(Z)-\psi(Y)$ cannot be much larger than the quantiles of $Z'\left\vert \nabla \psi(Z)\right\vert$.

This brings us to the problem of estimating the quantiles of $Z'\left\vert \nabla \psi(Z)\right\vert$. The gradient is a sum of random variables and one can use the triangle inequality, H\"{o}lder's inequality or a different duality to write it in terms of a sum of independent random variables. We then apply Propositions \ref{tool sum order stat v1}, \ref{tool sum order stat v2} and \ref{Rr3} below to handle such sums.

\subsection{Small lemmas}

\begin{lemma}\label{no effect no}
Let $\delta>0$ and let $Q$ be a random variable taking values in $[0,\infty)$ such that for all $u,v\in[0,\infty)$,
\[
\mathbb{P}\left\{Q\geq uv\right\}\geq \delta e^{-v^2/4}\mathbb{P}\left\{Q\geq u\right\}
\]
Let $Z'$ be a random variable with the standard normal distribution in $\mathbb{R}$. Then for all $t>0$,
\[
\mathbb{P}\left\{\left\vert Z' \right\vert Q\geq t\right\}\leq C\delta^{-1}\mathbb{P}\left\{Q\geq ct\right\}
\]
\end{lemma}

\begin{proof}
\begin{eqnarray*}
\mathbb{P}\left\{\left\vert Z' \right\vert Q\geq t\right\}&=&\frac{2}{\sqrt{2\pi}}\int_0^\infty e^{-s^2/2}\mathbb{P}\left\{Q\geq \frac{t}{s}\right\}ds\leq\frac{2}{\sqrt{2\pi}}\int_0^\infty e^{-s^2/4}\delta^{-1}\mathbb{P}\left\{Q\geq t\right\}ds
\end{eqnarray*}
\end{proof}

The following lemma will be used implicitly several times.

\begin{lemma}
\label{lil alg}If $f,g:\left[ 0,\infty \right) \rightarrow \left[ 0,\infty
\right) $ are continuous strictly increasing functions with $f\left(
0\right) =g\left( 0\right) =0$, $t\in \left[ 0,\infty \right) $ and $s=\max
\left\{ f\left( t\right) ,g\left( t\right) \right\} $ then $t=\min \left\{
f^{-1}\left( s\right) ,g^{-1}\left( s\right) \right\} $.
\end{lemma}

\subsection{\label{toolkit}Tools from \cite{Fr20IIa}, \cite{Fr conv maj} and \cite{Pis0}}

\begin{proposition}\label{tool sum order stat v1}(see proof of \cite[Lemma 8]{Fr20IIa})
Let $n\in \mathbb{N}$, $\lambda \in \left[ 2,\infty \right) $,
and let $(Y_{i})_{1}^{n}$
be an i.i.d. sequence of non-negative random variables, each with cumulative
distribution $F$, quantile function $F^{-1}$, and corresponding order statistics $\left( Y_{(i)}\right)
_{1}^{n}$. With probability at least $1-3^{-1}\pi ^{2}\exp \left( -\lambda
^{2}/2\right) $, the following event holds: for all $j,k\in\mathbb{Z}$ with $0\leq j\leq k<n$,
\begin{eqnarray*}
\sum_{i=n-k}^{n-j}Y_{(i)}&\leq&F^{-1}\left(
1-\frac{j+1}{n+1}e^{-1}\exp \left( \frac{-\lambda ^{2}-4\log \left(
j+1\right) }{2\left(j+1\right)}\right) \right)\\
&+&(n+1)\int_{(j+1)/\left( n+1\right) }^{(k+1)/\left( n+1\right) }F^{-1}\left(
1-e^{-1-2/e}t\exp \left( \frac{-\lambda ^{2}}{2\left( n+1\right) t}\right)
\right) dt
\end{eqnarray*}
\end{proposition}

\begin{proof}[Proof sketch]
This is based on concentration of order statistics from the uniform distribution on $(0,1)$ which in turn is based on estimates for the binomial distribution. One then transforms the order statistics in $(0,1)$ to the order statistics $(Y_{(i)})_1^n$ using the quantile function. One then peels off the largest term and bounds $\sum_{i=n-k}^{n-j-1}Y_{(i)}$ above by a deterministic sum (with high probability) and compares the deterministic sum to an integral.
\end{proof}

\begin{proposition}\label{tool sum order stat v2}(\cite[Proposition 10]{Fr20IIa})
Consider the setting and assumptions of Proposition \ref{tool sum order stat v1}, and assume in addition that there exists $p>0$ and $T\geq 1$ such that for all $\delta,x\in(0,1)$,
\begin{equation}
H^*(\delta x)\geq T^{-1} \delta^{-1/p}H^*(x)\label{condit fast grw}
\end{equation}
where $H^*(x)=F^{-1}(1-x)$. With probability at least $1-3^{-1}\pi ^{2}\exp \left( -\lambda
^{2}/2\right) $, the following event holds: for all $j,k\in\mathbb{Z}$ with $0\leq j\leq k<n$, $\sum_{i=n-k}^{n-j}Y_{(i)}$ is bounded above by
\begin{eqnarray*}
\left[1+T\lambda^2A\right]H^*\left(e^{-1-2/e}\frac{j+1}{n+1}\exp\left(\frac{-\lambda^2}{2(j+1)}\right)\right)+Cn\int_{\frac{j+1}{n+1}\exp\left(\frac{-\lambda^2}{2(j+1)}-1-2/e\right)}^{\frac{k+1}{n+1}\exp\left(\frac{-\lambda^2}{2(k+1)}-1-2/e\right)}H^*(x)dx
\end{eqnarray*}
where $A=0$ if $\lambda^2/2\leq j+1$ and $A$ equals
\begin{eqnarray*}
&&C^{1+1/p}\min\left\{p,\lambda^2\left(\frac{1}{j+1}-\frac{1}{\min\left\{\lambda^2/2,k+1\right\}}\right)\right\}\left(p+1+\frac{\lambda^2}{j+1}\right)^{-2}\\
&&+C\min\left\{1,\log\frac{\min\left\{k+1,\lambda^2/2\right\}}{j+1}\right\}\left[\frac{\lambda^2}{2(j+1)}\exp\left(\frac{\lambda^2}{2(j+1)}\right)\right]^{-1/p}\left[1+\frac{\lambda^2}{k+1}\right]^{-1}
\end{eqnarray*}
if $\lambda^2/2>j+1$.
\end{proposition}

\begin{proof}[Proof sketch]
One simplifies the integral in Proposition \ref{tool sum order stat v1} by making a change of variables and proving various expressions for the resulting integrals that arise. While the details are at times mildly tedious, the ingredients are mostly standard and idea is simple.
\end{proof}

Let $(W_i)_1^n$ be an i.i.d. sequence of non-negative random variables such that for all $t>0$,
\[
\mathbb{P}\left\{W_i >t\right\}=e^{q/2}(e+t)^{-q/2}\left(\ln (e+t)\right)^{q/2}
\]

\begin{proposition}\label{Rr3}(\cite[Proposition 21]{Fr20IIa})
For all $b\in[0,\infty)^n$ and all $t>0$, with probability at least $1-Ce^{-t^2/2}$,
\begin{equation*}\label{Rr3 linear combo}
\sum_{i=1}^nb_iW_i\leq C_q\left(\left\vert b\right\vert _1+t^2e^{t^2/q}\left\vert b\right\vert_{q/2}\right)
\end{equation*}
\end{proposition}

\begin{proof}[Proof sketch]
One defines a norm on $\mathbb{R}^n$, $[\cdot]_{\delta,W}$, which essentially measures the quantiles of $\sum_1^n\left\vert x_i\right\vert W_i$. One estimates $[x]_{\delta,W}$ for all $x\in\{0,1\}^n$, and by a duality argument (expressing this norm as the dual norm of its dual norm) one gets a bound on $[x]_{\delta,W}$ for all $x\in\mathbb{R}^n$. This estimate is good for most directions, but for certain directions becomes crude. One then compares $\sum_1^n\left\vert x_i\right\vert W_i$ to $\sum_1^n\left\vert x_{I(i)}\right\vert W_i$, where $(I(i))_1^n$ is an i.i.d. sequence of random integers in $\{1, 2, \cdots ,n\}$. The advantage of this is that the latter sum is a sum of i.i.d. random variables, so one can use Proposition \ref{tool sum order stat v2}. This second approach also works for most coefficient sequences and becomes crude in some cases. When we combine the first approach, which is more geometric, with the second approach, which is more combinatorial, we end up with Proposition \ref{Rr3}.
\end{proof}

\begin{proposition}\label{PisGau}(Pisier \cite{Pis0})
Let $Y$ and $Z$ be independent random vectors in $\mathbb{R}^n$ each with the standard normal distribution, and let $Z'$ be a random variable with the standard normal distribution in $\mathbb{R}$ independent of $Z$. Let $\psi:\mathbb{R}^n\rightarrow\mathbb{R}$ be differentiable on $\mathbb{R}$. For all convex functions $\varphi:\mathbb{R}\rightarrow\mathbb{R}$,
\begin{equation}
\mathbb{E}\varphi \left( \psi\left( Z^*\right) -\psi\left( Z\right) \right) \leq 
\mathbb{E}\varphi \left( \frac{\pi }{2}\left\vert \nabla \psi\left( Z\right)
\right\vert Z'\right)   \label{Pisiers Gaussian}
\end{equation}%
\end{proposition}

\begin{proof}[Proof sketch]
This follows by an elementary analysis using Jensen's inequality of the corresponding path integral from $Y$ to $Z$, for a carefully chosen path.
\end{proof}

\begin{proposition}\label{Gaussian vvv}(\cite[Corollary 6]{Fr conv maj})
Let $X$ and $Y$ be real valued random variables such that $\mathbb{E}\max\{0,Y\}<\infty$ and for all $\alpha\in\mathbb{R}$, $\mathbb{E}\max\{0,X-\alpha\}\leq\mathbb{E}\max\{0,Y-\alpha\}$. Suppose that $T,\gamma\geq 1$, $p>1$, and $Q:(0,\infty)\rightarrow(0,\infty)$ is a function that satisfies
\begin{equation*}
Q(t)\exp\left(\frac{-t^2}{2p}\right)\leq TQ(s)\exp\left(\frac{-s^2}{2p}\right)\label{Q inc rate}
\end{equation*}
for all $0<s<t$. If $\mathbb{P}\left\{Y>Q(t)\right\}< \gamma\exp\left(-t^2/2\right)$ for all $t>0$, then for all $t>0$
\begin{equation*}
\mathbb{P}\left\{X>\frac{pT}{p-1}Q(t)\right\}\leq \gamma\exp\left(-t^2/2\right)\label{gaussian v bound X}
\end{equation*}
\end{proposition}

\begin{proof}[Proof sketch]
This is based on the result of Meilijson and N\'{a}das \cite{MeiNad}, that if $X$ and $Y$ are real valued random variables with $\mathbb{E}\max\{0,Y\}<\infty$ and $\mathbb{E}\varphi(X)\leq\mathbb{E}\varphi(Y)$ for all non-decreasing convex $\varphi:\mathbb{R}\rightarrow[0,\infty)$, then for all $s\in\mathbb{R}$ with $\mathbb{P}\{Y>s\}\neq 0$, $\mathbb{P}\{X\geq\mathbb{E}(Y|Y>s)\}\leq\mathbb{P}\{Y>s\}$. Unless $Y$ has very thick tails, $\mathbb{E}(Y|Y>s)$ is the same order of magnitude as $s$; this is quantified in the function $Q$.
\end{proof}

\subsection{\label{Waii}Proof of Theorem \protect\ref{conc sum weibull type}}

Without loss of generality we may assume that each $F_{i}^{-1}$ is
differentiable. Let $Tx=\left( \left( F_{i}^{-1}\Phi (x_{i})\right)
_{i=1}^{n}\right) $, and let $Z$ be a random vector in $\mathbb{R}^{n}$ with
the standard normal distribution. $f_i$ denotes the partial derivative of $f$ with respect to the $i^{th}$ coordinate. By the assumed bound on $Lip\left( F_{i}^{-1}\Phi
,Z_{i}\right)$ and the triangle inequality,
\begin{eqnarray}
\left\vert \nabla (f\circ T)(Z)\right\vert &=&\left(
\sum_{i=1}^{n}f_{i}\left( TZ\right) ^{2}Lip\left( F_{i}^{-1}\Phi
,Z_{i}\right) ^{2}\right) ^{1/2}\nonumber\\
&\leq& \left( \sum_{i=1}^{n}f_{i}\left(
TZ\right) ^{2}\left( 1+\left\vert Z_{i}\right\vert \right) ^{-2+4/q}\right)
^{1/2} \label{grad expressi}\\
&\leq &\left( \sum_{i=1}^{n}\sup_{x\in \mathbb{R}^{n}}\left\vert f_{i}\left(
x\right) \right\vert ^{2}\mathbb{E}\left( 1+\left\vert Z_{i}\right\vert
\right) ^{-2+4/q}\right) ^{1/2} \nonumber\\
&&+\left\vert \sum_{i=1}^{n}\sup_{x\in \mathbb{R}^{n}}\left\vert f_{i}\left(
x\right) \right\vert ^{2}\left\{ \left( 1+\left\vert Z_{i}\right\vert
\right) ^{-2+4/q}-\mathbb{E}\left( 1+\left\vert Z_{i}\right\vert \right)
^{-2+4/q}\right\} \right\vert ^{1/2}\nonumber
\end{eqnarray}%
Since $0<q<1$, for all $t>0$,%
\begin{equation*}
\mathbb{P}\left\{ \left\vert \left( 1+\left\vert Z_{i}\right\vert \right)
^{-2+4/q}-\mathbb{E}\left( 1+\left\vert Z_{i}\right\vert \right)
^{-2+4/q}\right\vert >t\right\} \leq 2\exp \left( -c_{q}t^{q/(2-q)}\right)
\end{equation*}%
Noting that $0<2/(2-q)<1$ and using (\ref{Weibul dist for 0<q<=1}), with
probability at least $1-2\exp \left( -t^{2}/2\right) $,%
\begin{eqnarray*}
&&\left\vert \sum_{i=1}^{n}\sup_{x\in \mathbb{R}^{n}}\left\vert f_{i}\left(
x\right) \right\vert ^{2}\left\{ \left( 1+\left\vert Z_{i}\right\vert
\right) ^{-2+4/q}-\mathbb{E}\left( 1+\left\vert Z_{i}\right\vert \right)
^{-2+4/q}\right\} \right\vert ^{1/2} \\
&\leq &C_{q}\left( t^{1/2}Lip_{4}^{\sharp }\left( f\right)
+t^{(2-q)/q}Lip_{\infty }^{\sharp }\left( f\right) \right)
\end{eqnarray*}%
When adding this to the remaining term of $C_{q}Lip_{2}^{\sharp }\left(
f\right) $, the term involving $Lip_{4}^{\sharp }\left( f\right) $ can be
erased since by H\"{o}lder's inequality $\left\vert \cdot \right\vert
_{4}\leq \left\vert \cdot \right\vert _{2}^{1/2}\left\vert \cdot \right\vert
_{\infty }^{1/2}$ and since we may assume that $t\geq 1$, $t^{1/2}\leq
\left( 1\right) ^{1/2}\left( t^{(2-q)/q}\right) ^{1/2}$. What we have shown
is that with probability at least $1-2\exp \left( -t^{2}/2\right) $,%
\begin{equation*}
\left\vert \nabla (f\circ T)(Z)\right\vert \leq C_{q}Lip_{2}^{\sharp }\left(
f\right) +C_{q}t^{(2-q)/q}Lip_{\infty }^{\sharp }\left( f\right)
\end{equation*}
and therefore with probability at least $1-C\exp \left( -t^{2}/2\right)$,
\begin{equation*}
\left\vert Z'\right\vert\cdot\left\vert \nabla (f\circ T)(Z)\right\vert \leq C_{q}tLip_{2}^{\sharp }\left(
f\right) +C_{q}t^{2/q}Lip_{\infty }^{\sharp }\left( f\right)
\end{equation*}
As in Section \ref{outi}, $Y$ is an independent copy of $Z$, and $Z'$ is a real valued variable with the standard normal distribution, independent of $Y$ and $Z$. It now follows from Pisier's version of the Gaussian concentration inequality that for all non-decreasing convex $\varphi:\mathbb{R}\rightarrow\mathbb{R}$,
\[
\mathbb{E}\varphi\left((f\circ T)(Z)-(f\circ T)(Y)\right)\leq\mathbb{E}\varphi\left(Z'\left\vert \nabla (f\circ T)(Z)\right\vert \right)\leq\mathbb{E}\varphi\left(\left\vert Z'\right\vert \cdot\left\vert \nabla (f\circ T)(Z)\right\vert \right)
\]
By \cite[Proposition 6]{Fr conv maj}, the quantiles of $\left(f\circ T\right)(Z)-\left(f\circ T\right)(Y)$ cannot be much larger than those of $\left\vert Z' \right\vert \cdot\left\vert \nabla \left( f\circ T\right) \left( Z\right) \right\vert$ (this should not be surprising, even if the reader is not familiar with \cite{Fr conv maj}). More precisely, with probability at least $C\exp \left( -t^{2}/2\right) $,
\[
\left\vert \left(f\circ T\right)(Z)-\left(f\circ T\right)(Y) \right\vert \leq  C_{q}tLip_{2}^{\sharp }\left(
f\right) +C_{q}t^{2/q}Lip_{\infty }^{\sharp }\left( f\right)
\]
Here we are also using symmetry of the distribution of $\left(f\circ T\right)(Z)-\left(f\circ T\right)(Y)$. It is well known, and an easy computation, that concentration about an independent random point implies concentration about the median by changing the constants involved (this is a more robust phenomenon than a comparison to concentration about the mean). This proves (\ref{weibull variation}).

For $r\geq 1$ consider the following norm on $\mathbb{R}%
^{n}$,%
\begin{equation*}
\left\vert x\right\vert _{B_{1}^{n}\cap r^{-1}B_{\infty }^{n}}=\max \left\{
\left\vert x\right\vert _{1},r\left\vert x\right\vert _{\infty }\right\}
\end{equation*}%
After a brief consideration we see that the dual norm is given by%
\begin{eqnarray*}
\left\vert y\right\vert _{conv\left( B_{\infty }^{n},rB_{1}^{n}\right)
}&=&\max \left\{ \sum_{i=1}^{n}x_{i}y_{i}:\left\vert x\right\vert _{1}\leq
1,\left\vert x\right\vert _{\infty }\leq r^{-1}\right\}\\
&\leq&r^{-1}\sum_{i=1}^{\min \left\{ \left\lceil r\right\rceil ,n\right\}
}y^{(i)}\leq 2\left\vert y\right\vert _{conv\left( B_{\infty
}^{n},rB_{1}^{n}\right) }
\end{eqnarray*}%
where $y^{(1)}\geq y^{(2)}\ldots $ is the non-increasing rearrangement of $%
\left( \left\vert y_{i}\right\vert \right) _{1}^{n}$. Set $r=t^{-2+4/q}$, $k=\left\lceil r\right\rceil $ and $Y_i=\left(1+\left\vert Z_i \right\vert\right)^{-2+4/q}$. Each $Y_i$ has quantile function $H^*(1-x)$, where
\[
H^*(x)\leq C_q\left(1+\log \frac{1}{x}\right)^{-1+2/q}
\]
By Proposition \ref{tool sum order stat v1}, with probability at least $1-C\exp \left(
-t^{2}/2\right) $, (assuming first that $t^{-2+4/q}\leq c_{q}n$),%
\begin{eqnarray*}
\sum_{i=1}^kY^{(i)}=\sum_{i=n-k+1}^nY_{(i)}\leq H^*\left(\frac{c}{n}\exp\left(\frac{-t^2}{2}\right)\right)+Cn\int_{1/(n+1)}^{k/(n+1)}H^*\left(cx\exp\left(\frac{-t^2}{2(n+1)x}\right)\right)dx
\end{eqnarray*}
\begin{eqnarray*}
&\leq&C_q\left(t^2+\log n\right)^{-1+2/q}+C_qn\frac{k-1}{n+1}+C_qn\int_{1/(n+1)}^{k/(n+1)}\left(\log\frac{1}{x}\right)^{-1+2/q}dx\\
&+&C_qt^{-2+4/q}n^{2-2/q}\int_{1/(n+1)}^{k/(n+1)}x^{1-2/q}dx\\
&\leq& C_qt^{-2+4/q}+C_q\left(\log n\right)^{-1+2/q}+C_qk+C_qk\left(\log\frac{n+1}{k}\right)^{-1+2/q}+C_qt^{-2+4/q}k^{2-2/q}
\end{eqnarray*}
Recalling the relationships between $t$, $k$ and $n$, and the fact that $3-2/q<1$, this is bounded above by
\[
C_qk\left(\log\frac{n+1}{k}\right)^{-1+2/q}
\]
This implies
\begin{equation*}
\left\vert \left( \left( 1+\left\vert Z_{i}\right\vert \right)
^{-2+4/q}\right) _{1}^{n}\right\vert _{conv\left( B_{\infty
}^{n},rB_{1}^{n}\right) }\leq C_{q}\left( \log \frac{n+1}{k}\right)
^{-1+2/q}
\end{equation*}%
For $t^{-2+4/q}>c_{q}n$, it follows by classical Gaussian concentration of
the $\ell _{-2+4/q}^{n}$ norm (which is $1$-Lipschitz since $-2+4/q\geq 2$)
that with probability at least $1-C\exp \left( -t^{2}/2\right) $, 
\begin{equation*}
\left\vert \left( \left( 1+\left\vert Z_{i}\right\vert \right)
^{-2+4/q}\right) _{1}^{n}\right\vert _{conv\left( B_{\infty
}^{n},rB_{1}^{n}\right) }\leq C_{q}n^{-1}\sum_{i=1}^{n}\left( 1+\left\vert
Z_{i}\right\vert \right) ^{-2+4/q}\leq C_{q}+C_{q}t^{-2+4/q}n^{-1}\leq C_{q}
\end{equation*}%
In either case,
\[
\left\vert \left( \left( 1+\left\vert Z_{i}\right\vert \right)
^{-2+4/q}\right) _{1}^{n}\right\vert _{conv\left( B_{\infty
}^{n},rB_{1}^{n}\right) }\leq C_{q}\left( \log \left(e+ \frac{n}{t^{-2+4/q}}\right)\right)
^{-1+2/q}
\]
Applying duality to the sum in (\ref{grad expressi}), 
\begin{equation*}
\left\vert \nabla \left( f\circ T\right) \left( Z\right) \right\vert
^{2}\leq \left\vert \left( f_{i}\left( TZ\right) ^{2}\right)
_{1}^{n}\right\vert _{B_{1}^{n}\cap r^{-1}B_{\infty }^{n}}\left\vert \left(
\left( 1+\left\vert Z_{i}\right\vert \right) ^{-2+4/q}\right)
_{1}^{n}\right\vert _{conv\left( B_{\infty }^{n},rB_{1}^{n}\right) }
\end{equation*}%
so
\[
\left\vert \nabla \left( f\circ T\right) \left( Z\right) \right\vert\leq C_q\left( \log \left(e+ \frac{n}{t^{-2+4/q}}\right)\right)
^{1/q-1/2}\left(Lip_{2}\left(
f\right) +t^{(2-q)/q}Lip_{\infty }\left( f\right)\right)
\]
(\ref{nonlinear Weibull 2}) now follows, as earlier in the proof, from Pisier's version of Gaussian concentration and applying tail comparison under convex majorization.

\subsection{Proof of Theorem \protect\ref{poly tails Lip}}

Without loss of generality we may assume that each $F_{i}^{-1}$ is
differentiable, and that $Lip_{2}^{\sharp }\left( f\right) =1$. Let $%
Tx=\left( F_{i}^{-1}\Phi (x_{i})\right) _{i=1}^{n}$ and let $Z$ be a random
vector in $\mathbb{R}^{n}$ with the standard normal distribution. Using the estimate
\[
\frac{\phi(x)}{\min\left\{\Phi(x),1-\Phi(x)\right\}}\leq C\left(1+\left\vert x \right\vert\right)\leq C\sqrt{\log \min \left\{ \Phi \left( x\right)
,1-\Phi \left( x\right) \right\} ^{-1}}
\]
valid for all $x\in\mathbb{R}$, which follows by integrating the tails of $\phi$ using log-concavity, and using
\begin{eqnarray*}
Lip\left(F_i^{-1}\circ\Phi,x\right)&\leq& Lip\left(F_i^{-1},\Phi(x)\right)\cdot Lip\left(\Phi,x\right)\\
&\leq&\min\left\{\Phi(x),1-\Phi(x)\right\}^{-1/q}\frac{\phi(x)}{\min\left\{\Phi(x),1-\Phi(x)\right\}}
\end{eqnarray*}
we see that
\begin{equation*}
\left\vert \nabla \left( f\circ T\right) \left( Z\right) \right\vert =\left(
\sum_{i=1}^{n}f_{i}\left( TZ\right) ^{2}Lip\left( F_{i}^{-1}\Phi
,Z_{i}\right) ^{2}\right) ^{1/2}\leq \left( \sum_{i=1}^{n}f_{i,\sharp
}^{2}U_{i}\right) ^{1/2}
\end{equation*}%
where%
\begin{equation}
f_{i,\sharp }=\sup_{x\in \mathbb{R}^{n}}\left\vert f_{i}(x)\right\vert 
\hspace{0.45in}U_{i}=C\min \left\{ \Phi \left( Z_{i}\right) ,1-\Phi \left(
Z_{i}\right) \right\} ^{-2/q}\log \min \left\{ \Phi \left( Z_{i}\right)
,1-\Phi \left( Z_{i}\right) \right\} ^{-1}  \label{aanbd}
\end{equation}%
It follows from the definition that each $U_{i}$ has quantile function%
\begin{equation}
G^{-1}(t)=\left( \frac{1-t}{2}\right) ^{-2/q}\log \left( \frac{1-t}{2}%
\right) ^{-1}:0<t<1  \label{aanber}
\end{equation}%
By Proposition \ref{Rr3},
\begin{equation*}
\mathbb{P}\left\{ \sum_{i=1}^{n}f_{i,\sharp }^{2}U_{i}>C_{q}\sum_{i=1}^{n}f_{i,\sharp }^{2}+C_qt^2e^{t^2/q}\left(\sum_{i=1}^{n}f_{i,\sharp }^{q}\right)^{q/2}\right\} \leq Ce^{-t^2/2}
\end{equation*}%
Recalling the expression for $\left\vert \nabla \left( f\circ T\right) \left( Z\right) \right\vert$ in terms of the $U_i$: with probability at least $1-Ce^{-t^{2}/2}$,
\begin{equation*}
\left\vert \nabla \left( f\circ T\right) \left( Z\right) \right\vert \leq
C_{q}\left( Lip_{2}^{\sharp }\left( f\right) +te^{t^2/(2q)} Lip_{q}^{\sharp }\left( f\right)\right)
\end{equation*}%
This means we can write
\[
\left\vert \nabla \left( f\circ T\right) \left( Z\right) \right\vert =C_{q} Lip_{2}^{\sharp }\left( f\right) +C_{q}G Lip_{q}^{\sharp }\left( f\right)
\]
where $G$ is a random variable with heavy tails (polynomial type heaviness). Upon multiplication, a random variable with sufficiently heavy tails absorbs a standard normal random variable, in that the quantiles of the product are the same order of magnitude. This is made precise, for example, in Lemma \ref{no effect no}. If $Z'$ is a standard normal random variable in $\mathbb{R}$ independent of $Z$, then this observation means that with probability at least $1-Ce^{-t^{2}/2}$,
\[
\left\vert Z' \right\vert C_{q} Lip_{2}^{\sharp }\left( f\right)\leq C_{q} tLip_{2}^{\sharp }\left( f\right)\hspace{1cm}\text{and}\hspace{1cm}\left\vert Z' \right\vert C_{q}G Lip_{q}^{\sharp }\left( f\right)\leq C_{q}te^{t^2/(2q)} Lip_{q}^{\sharp }\left( f\right)
\]
so
\begin{equation*}
\left\vert Z' \right\vert\cdot\left\vert \nabla \left( f\circ T\right) \left( Z\right) \right\vert \leq
C_{q}t\left( Lip_{2}^{\sharp }\left( f\right) +e^{t^2/(2q)} Lip_{q}^{\sharp }\left( f\right)\right)
\end{equation*}%
Eq. (\ref{poly conc one}) now follows, as in the proof of Theorem \ref{conc sum weibull type}, from Gaussian concentration and tail comparison under convex majorization. Here one uses Proposition \ref{Gaussian vvv} with, say, $T=C_q$ and $p=1.5$. We now consider Eq. (\ref{poly conc
two}). No longer assuming that $Lip_{2}^{\sharp }\left( f\right) =1$, using H%
\"{o}lder's inequality for $\ell _{p/2}$ and $\ell _{p/(p-2)}$,%
\begin{equation*}
\left\vert \nabla \left( f\circ T\right) \left( Z\right) \right\vert \leq
Lip_{p}\left( f\right) \left( \sum_{i=1}^{n}U_{i}^{p/(p-2)}\right)
^{(p-2)/(2p)}
\end{equation*}%
By Proposition \ref{tool sum order stat v2}, with probability at
least $1-C\exp \left( -t^{2}/2\right) $,%
\begin{equation*}
\left( \sum_{i=1}^{n}U_{i}^{p/(p-2)}\right) ^{(p-2)/(2p)}\leq C_{p,q}\left(
n^{1/2-1/p}+n^{1/q}te^{t^2/(2q)} \right)
\end{equation*}%
The result now follows, as before, by absorbing a Gaussian into a heavy tailed random variable, and applying Gaussian concentration and convex majorization.

\section*{Acknowledgements}

Thanks to Gusti van Zyl and the anonymous referee for comments and suggestions. Part of this work was done while the author was a postdoctoral
fellow at the Weizmann Institute of Science.


\begin{thebibliography}{99}

\bibitem{BaCatRo} Barthe, F., Cattiaux, P., Roberto C.: Concentration for
independent random variables with heavy tails. AMRX, Appl. Math. Res.
Express (2), 39-60 (2005)

\bibitem{CatGoGuRo} Cattiaux, P., Gozlan, N., Guillin A., Roberto, C.:
Functional inequalities for heavy tailed distributions and application to
isoperimetry. Electon. J. Probab. 15 (13), 346-385 (2010)

\bibitem{CorLed} Cordero-Erausquin, D., Ledoux, M.: Hypercontractive
measures, Talagrand's inequality, and influences. Geometric aspects of
functional analysis, 169-189, Lecture Notes in Math., 2050, Springer (2012)

\bibitem{Fr20} Fresen, D. J: Variations and extensions of the Gaussian
concentration inequality, Part I. To appear in Quaest. Math. Published online at https://www.tandfonline.com/doi/abs/10.2989/16073606.2022.2074908. Preprint available at https://arxiv.org/abs/1812.10938

\bibitem{Fr20IIa} Fresen, D. J: Deviation inequalities for linear combinations of random variables with power tails. arXiv: 2207.01867 (results cited with numbering as in v3, v4 to be released soon with updated title)

\bibitem{Fr22} Fresen, D. J: Random Euclidean embeddings in finite dimensional Lorentz spaces. To appear. arXiv:2104.11974

\bibitem{Fr conv maj} Fresen, D. J: Optimal tail comparison under convex majorization. arXiv: 2207.01872.

\bibitem{Gozl} Gozlan, N.: Poincar\'{e} inequalities and dimension free
concentration of measure. Ann. Inst. Henri Poincar\'{e} Probab. Stat. 46
(3), 708-739 (2010)

\bibitem{HitMont1997} Hitczenko, P., Montgomery-Smith, S. J., Oleszkiewicz,
K.: Moment inequalities for sums of certain independent symmetric random
variables. Studia Math. 123 (1), 15-42 (1997)

\bibitem{Lat97} Lata\l a, R.: Estimates of moments of sums of independent
real random variables. Ann. Probab. 25 (3), 1502-1513 (1997)

\bibitem{MeiNad} Meilijson, I., N\'{a}das, A.: Convex Majorization with an Application to the Length of Critical Paths. J. Appl. Prob. 16 (3), 671-677 (1979)

\bibitem{Pet75} Petrov, V. V.: Sums of independent random variables.
Springer-Verlag 1975 (translated from Russian into English)

\bibitem{Pis0} Pisier, G.: Probabilistic methods in the geometry of Banach
spaces. CIME, Varenna, 1985. Lecture Notes in Mathematics 1206, 167-241
(1986)

\bibitem{Tal94sup} Talagrand, M.: The supremum of some canonical processes.
Amer. J. Math. 116 (2), 283-325 (1994)



\end{thebibliography}
\end{document}